\documentclass{amsart}
\usepackage{amssymb}
\usepackage{hyperref}
\usepackage{graphicx}
\usepackage{curves}

\newcommand*{\mailto}
[1]{\href{mailto:#1}{\nolinkurl{#1}}}
\newtheorem{theorem}{Theorem}[section]

\newtheorem{remark}[theorem]{Remark}

\newcommand{\R}{\mathbb{R}}

\newcommand{\be}{\begin{equation}}
\newcommand{\ee}{\end{equation}}
\newcommand{\bea}{\begin{eqnarray}}
\newcommand{\eea}{\end{eqnarray}}

\newcommand{\im}{\mathop{\mathrm{Im}}}

\def\XXint#1#2#3{{\setbox0=\hbox{$#1{#2#3}{\int}$}
		\vcenter{\hbox{$#2#3$}}\kern-.5\wd0}}

\numberwithin{equation}{section}

\input{tcilatex}

\begin{document}
\title[KdV equation]{A new asymptotic regime for the KdV equation with
Wigner-von Neumann type initial data}
\author{Alexei Rybkin}
\address{Department of Mathematics and Statistics, University of Alaska
Fairbanks, PO Box 756660, Fairbanks, AK 99775}
\email{arybkin@alaska.edu}
\thanks{The author is supported in part by the NSF under grant DMS-2307774.}
\date{September 2025}

\begin{abstract}
We investigate the long-time asymptotic behavior of solutions to the Cauchy
problem for the KdV equation, focusing on the evolution of the radiant wave
associated with a Wigner-von Neumann (WvN) resonance induced by the initial
data (potential). A WvN resonance refers to an energy level where the
potential exhibits zero transmission (complete reflection). The
corresponding Jost solution at such energy becomes singular, and in the NLS
context, this is referred to as a spectral singularity. A WvN resonance
represents a long-range phenomenon, often introducing significant
challenges, such as an infinite negative spectrum, when employing the
inverse scattering transform (IST). To avoid some of these issues, we
consider a restricted class of initial data that generates a WvN resonance
but for which the IST framework can be suitably adapted. For this class of
potentials, we demonstrate that each WvN resonance produces a distinct
asymptotic regime---termed the resonance regime---characterized by a slower
decay rate for large time compared to the radiant waves associated with
short-range initial data.
\end{abstract}

\maketitle

\section{Introduction}

We are concerned with extension of the powerful nonlinear steepest descent
method due to Deift and Zhou \cite{DeiftZhou1993} (aka the Riemann-Hilbert
problem approach) for the study of the long-time behavior of the solution to
the Cauchy problem for the Korteweg-de Vries (KdV) equation (see e.g. \cite%
{AC91},\cite{NovikovetalBook})%
\begin{equation}
\partial _{t}q-6q\partial _{x}q+\partial _{x}^{3}q=0,  \label{KdV}
\end{equation}%
\begin{equation}
q\left( x,0\right) =q\left( x\right) ,  \label{KdVID}
\end{equation}%
with certain slowly decaying (long-range) initial profiles $q$. Note that 
\cite{DeiftZhou1993} deals with the modified KdV (mKdV) for $q\left(
x\right) $ from the Schwarz class, which of course assumes rapid decay. A
comprehensive treatment of the KdV case is offered in \cite{GT09} where
decay and smoothness assumptions are relaxed to%
\begin{equation}
\int_{\mathbb{R}}\left( 1+\left\vert x\right\vert \right) \left( \left\vert
q\left( x\right) \right\vert +\left\vert q^{\prime }\left( x\right)
\right\vert +\left\vert q^{\prime \prime }\left( x\right) \right\vert
+\left\vert q^{\prime \prime \prime }\left( x\right) \right\vert \right) 
\mathrm{d}x<\infty \text{.}  \label{sr}
\end{equation}

It seems plausible (but we do not have any reference)\ that conditions on
the derivatives may be further relaxed at the expense of technical
complications but the condition

\begin{equation}
\int_{\mathbb{R}}\left( 1+\left\vert x\right\vert \right) \left\vert q\left(
x\right) \right\vert \mathrm{d}x<\infty\text{ \ (short-range)}  \label{sr 1}
\end{equation}
is crucially important to run the underlying inverse scattering transform\
(IST) (see e.g. \cite{March86}). In our \cite{GruRybSIMA15} we extend the
IST to $q$'s with essentially arbitrary behavior at $-\infty$ \footnote{%
Which means that (\ref{KdV}) does not require a boundary condition at $%
-\infty$ .} but still require a sufficiently fast decay\ at $+\infty$. This
should not come as a surprise since the KdV is a strongly unidirectional
equation (solitons run to the right and radiation waves run to the left)
which should translate into different contributions from the behavior of the
data $q$ at $\pm\infty$. We emphasize that relaxation of the condition (\ref%
{sr 1}) leads to a multitude of serious complications that can be resolved
only in some particular cases and not surprisingly long-time asymptotic
behavior of the corresponding KdV solutions is out of reach in general. Our
goal is to shed some light on this difficult situation but we need first to
fix our terminology.

Associate with $q$ the full line (self-adjoint) Schrodinger operator $%
\mathbb{L}_{q}=-\partial _{x}^{2}+q(x)$. For its spectrum $\sigma \left( 
\mathbb{L}_{q}\right) $ we have%
\begin{equation*}
\sigma \left( \mathbb{L}_{q}\right) =\sigma _{\mathrm{d}}\left( \mathbb{L}%
_{q}\right) \cup \sigma _{\mathrm{ac}}\left( \mathbb{L}_{q}\right) \text{,}
\end{equation*}%
where the discrete component $\sigma _{\mathrm{d}}\left( \mathbb{L}%
_{q}\right) =\{-\kappa _{n}^{2}\}$ is finite and for the absolutely
continuous one has $\sigma _{\mathrm{ac}}\left( \mathbb{L}_{q}\right)
=[0,\infty )$. There is no positive singular continuous spectrum. The Schr{%
\"{o}}dinger equation 
\begin{equation}
\mathbb{L}_{q}\psi =k^{2}\psi  \label{SE}
\end{equation}%
has two (linearly independent) Jost solutions $\psi ^{\left( \pm \right)
}(x,k)$, i.e. solutions satisfying%
\begin{equation}
\psi ^{\left( \pm \right) }(x,k)\sim \mathrm{e}^{\pm \mathrm{i}%
kx},\;\partial _{x}\psi ^{\left( \pm \right) }(x,k)\sim \pm \mathrm{i}k\psi
^{\left( \pm \right) }(x,k),\text{\ \ }x\rightarrow \pm \infty ,
\label{Jost solutions}
\end{equation}%
that are analytic in the upper half-plane and continuous to the real line.
Since $q$ is real, $\overline{\psi ^{\left( +\right) }}$ also solves (\ref%
{SE}) and one can easily see that the pair $\{\psi ^{\left( +\right) },%
\overline{\psi ^{\left( +\right) }}\}$ forms a fundamental set for (\ref{SE}%
). Hence $\psi ^{\left( -\right) }$ is a linear combination of $\{\psi
^{\left( +\right) },\overline{\psi ^{\left( +\right) }}\}$. We write this
fact as follows ($k\in \mathbb{R}$)%
\begin{equation}
T(k)\psi ^{\left( -\right) }(x,k)=\overline{\psi ^{\left( +\right) }(x,k)}%
+R(k)\psi ^{\left( +\right) }(x,k),\text{ (basic scattering identity)}
\label{R basic scatt identity}
\end{equation}%
where $T$ and $R$ are called the\emph{\ transmission} and (right)\emph{\
reflection} coefficient respectively. The identity (\ref{R basic scatt
identity}) is totally elementary but serves as a basis for inverse
scattering theory. As is well-known (see, e.g. \cite{March86}), the triple%
\begin{equation}
S_{q}=\{R,(\kappa _{n},c_{n})\},  \label{SD}
\end{equation}%
where $c_{n}=\left\Vert \psi ^{\left( +\right) }(\cdot ,\mathrm{i}\kappa
_{n})\right\Vert ^{-1}$, determines $q$ uniquely and is called the
scattering data for $\mathbb{L}_{q}$ .

Among initial data (aka potentials) that are not subject to (\ref{sr})
Wigner-von Neumann (WvN) type potentials%
\begin{equation}
q\left( x\right) =(A/x)\sin \left( 2\omega x+\delta \right) +O\left(
x^{-2}\right) ,x\rightarrow \pm \infty ,  \label{pure WvN}
\end{equation}%
where $A,\omega ,\delta $ are real constants, play a particularly important
role due to their physical relevance. Note that $q$'s satisfying $\left( \ref%
{pure WvN}\right) $ are in $L^{2}$ and therefore, due to the famous
Bourgain's result \cite{Bourgain93}, the Cauchy problem for the KdV equation
(\ref{KdV})-(\ref{KdVID}) is well-posed. Moreover, the scattering theory can
be developed along the same lines with its short-range counterpart except.
The main feature of WvN type potentials is the important fact that if $%
\gamma =\left\vert A/\left( 4\omega \right) \right\vert >1/2$ then for some
boundary condition at $x=0$ the half-line Schrodinger operator with a
potential of the form $\left( \ref{pure WvN}\right) $ has a positive bound
state at energy $\omega ^{2}$. An explicit example to this effect was
constructed in the seminal paper by Wigner-von Neumann \cite{WvN1929}.
Certain $q$'s subject to $\left( \ref{pure WvN}\right) $ support a positive
(embedded) bound state in the full line context (see e.g. \cite{RybCMP23}
and the literature cited therein) but it is never the case for potentials
restricted to a half-line.\ If $\omega ^{2}$ is not an embedded bound state
then the energy $\omega ^{2}$ is commonly referred as to a \emph{WvN
resonance. }It is shown in \cite{Klaus91} that the Jost solutions $\psi
^{\left( \pm \right) }(x,k)$ blow up to the order $\gamma =\left\vert
A/\left( 4\omega \right) \right\vert $ at $\pm \omega $ and for this reason $%
\omega $ could be called a \emph{spectral singularity} of order $\gamma $,
the term that is more common in the soliton community. At such a point, $%
T\left( \pm \omega \right) =0$ and therefore \cite{March86} one of the
(necessary) condition guaranteeing that the triple (\ref{SD}) is indeed
scattering data fails. Our interest in WvN potentials is inspired in part by
the work of Matveev (see \cite{Mat02} and the literature cited therein) and
his proposal \cite{MatveevOpenProblems}: "A very interesting unsolved
problem is to study the large time behavior of the solutions to the KdV
equation corresponding to the smooth initial data like $cx^{-1}\sin 2kx$, $%
c\in \mathbb{R}$. Depending on the choice of the constant $c$ the related
Schr{\"{o}}dinger operator might have finite or infinite or zero number of
the negative eigenvalues. The related inverse scattering problem is not yet
solved and the study of the related large times evolution is a very
challenging problem." Note that what Matveev says about the negative
spectrum is in fact an old result by Klaus \cite{Klaus82}: the potential%
\begin{equation*}
q_{\gamma }=(A/x)\sin \left( 2\omega x+\delta \right)
\end{equation*}%
has finite negative spectrum if $\gamma =\left\vert A/\left( 4\omega \right)
\right\vert <\sqrt{1/2}$ but if $\gamma \geq \sqrt{1/2}$ the negative
spectrum (necessarily discrete) is infinite (accumulating to zero). Thus, if 
$\gamma \geq \sqrt{1/2}$ then the triple (\ref{SD}) also should include
infinitely many arbitrary norming constants $c_{n}$ (Recall that a norming
constant determines a soliton's initial location).

While the Matveev problem is still a long shot, we have some partial results
on the inverse scattering problem. Namely, in \cite{GruRybNON22} we show
that if $\gamma $ is small enough (small coupling constant) the problem (\ref%
{KdV})-(\ref{KdVID}) can be solved by the IST (in fact, for a linear
combination of WvN potentials). However, this good news for small WvN
potentials does not yet imply applicability of the nonlinear steepest
descent, which Matveev refers to as a "very challenging problem". The main
issue here is that WvN resonances cause the underlying Riemann-Hilbert
problem (RHP) to become singular (and the standard conjugation step of RHP
fails). Another issue is a poor understanding (as opposed to the classical
case) of smoothness properties of $R$ (at the origin in particular) which
complicates the contour deformation step in RHP.

In this contribution we solely focus on the effect of WvN resonances on the
KdV solution. We show that a WvN resonance gives rise to a new large-time
asymptotic region $x\sim -12\omega ^{2}t$ as $t\rightarrow +\infty $, which
we call the \emph{resonance regime}. To show this we adapt the nonlinear
steepest descent to handle a certain singular RHP. In other words, we find a
class of explicit initial data that produces one WvN resonance and is free
of the IST issues associated with the long range nature of our initial data.

Our approach rests on \cite{GT09}, \cite{kt}, and \cite{Budylin20}. We start
out from the standard well-posed vector RHP, which we conjugate with the
partial transmission coefficient. The latter lets us deform the conjugated
RHP to the one with the jump matrix that is exponentially close to the
identity matrix away from $\pm \omega $ (in fact, without loss of
generality, we can set $\omega =1$). This, in turn, allows us to reduce the
original RHP to the one on small crosses centered at $\pm \omega $, which
can be decoupled into two RHP at $+\omega $ and $-\omega $. The solution to
each of them reduces to the one centered at $0$. We then replace the entries
of the jump matrix on the small cross with their behaviors at $0$ followed
by extending the cross to infinity. So far we have followed \cite{GT09}, 
\cite{kt} with no change. However, since the partial transmission
coefficient is unbounded at $\pm \omega $, the jump matrix of our RHP on the
cross is unbounded (in \cite{GT09}, \cite{kt} it is bounded). It takes only
a simple adjustment to the conjugation step from \cite{GT09}, \cite{kt} to
reduce our singular RHP back to a matrix\footnote{%
At this point it is more convenient to go over to the matrix RHP.} RHP on
the real line, which we call the model matrix RHP. As opposed to the
classical situation of \cite{GT09}, \cite{kt}, the jump matrix of our model
RHP is not constant (and not even bounded) and cannot be explicitly solved
in terms of parabolic cylinder functions. The classical approach however
admits a modification which is suitably done in \cite{Budylin20}\footnote{%
Note that this paper does not use or even mention \cite{GT09}, \cite{kt}.}.
The results of \cite{Budylin20} then yield the asymptotic of the solution to
our model RHP (for large spectral variable), which is what is required to
complete the derivation of our asymptotic formula (Theorem \ref{Thm on WvN})
following again \cite{GT09}.

Since it is a case study, we do not give here full detail leaving this for a
future publication where our rather narrow class of potentials will be
placed in a much broader class.

The paper is organized as follows. In Section \ref{Sec Notation} we fix our
terminology and introduce our main ingredients. In Section \ref{sec Main
results} we state our main results. Section \ref{sec original RHP} is
devoted to introducing our original RHP. In Section \ref{sec Partial T} we
introduce the partial transmission coefficient and find its asymptotics
around critical points that are crucially important to what follows. Section %
\ref{sec:condef} is devoted to standard conjugation and deformation of our
original RHP. In Section \ref{sec Matrix RHP} we reduce our deformed RHP to
a new matrix RHP on two small crosses centered at the critical points. In
Section \ref{sec asym of B's} we find asymptotics at the critical points of
the jump matrix which allows us to reduce the RHP\ to a matrix one with the
jump matrix on a single but infinite cross in Section \ref{sec RHP on a
cross}. Section \ref{sec model RHP} deals with a second conjugation step
where we transform the problem on a cross back to the problem on the real
line. In Section \ref{sec add RHP} we solve the new problem on the real line
by transforming it to an additive RHP. In Section \ref{sec proof} we finally
prove Theorem \ref{Thm on WvN}. In the final section \ref{sec concl} we make
some remarks on what we have not done.

\section{Notation and Auxiliaries\label{Sec Notation}}

We follow standard notation accepted in complex analysis: $\mathbb{C}$ is
the complex plane, $\mathbb{C}^{\pm}=\left\{ z\in\mathbb{C}:\pm\limfunc{Im}%
z>0\right\} $, $\overline{z}$ is the complex conjugate of $z$.

Matrices (including rows and columns) are denoted by boldface letters;
low/upper case being reserved for $2\times 1$ (row vector) and $2\times 2$
matrix respectively with an exception for Pauli matrices which are
traditionally denoted by%
\begin{equation*}
\sigma _{1}=\left( 
\begin{array}{cc}
0 & 1 \\ 
1 & 0%
\end{array}%
\right) ,\ \ \ \sigma _{3}=\left( 
\begin{array}{cc}
1 & 0 \\ 
0 & -1%
\end{array}%
\right) .
\end{equation*}

We write $f\left( z\right) \sim g\left( z\right) ,\ $as $z\rightarrow z_{0}%
\text{,}$ if $\lim\left( f\left( z\right) -g\left( z\right) \right) =0,~$as $%
z\rightarrow z_{0}$. We use the standard big O notation when the rate needs
to be specified.

$\log z$ is always defined with a cut along $\left( -\infty ,0\right) $.
That is, $\func{Im}\log z=\arg z\in (-\pi ,\pi ]$.

Given an oriented contour $\Gamma $, by $z_{+}$ ($z_{-}$)\ we denote the
positive (negative) side of $\Gamma $. Recall that the positive (negative)
side is the one that lies to the left (right) as one traverses the contour
in the direction of the orientation. For a function $f\left( z\right) $
analytic in $\mathbb{C}\setminus \Gamma $, by $f_{+}(z),z\in \Gamma $, we
denote the angular limit from above and by $f_{-}(z)$ the one from below. In
particular, if $\Gamma =\mathbb{R}$ then $f_{\pm }\left( x\right) =f\left(
x\pm \mathrm{i}0\right) =\lim_{\varepsilon \rightarrow +0}f\left( x\pm 
\mathrm{i}\varepsilon \right) ,x\in \mathbb{R}$.

Finally, $\mathrm{1}_{S}$ is the characteristic function of a real set $S$.

\section{Our class of initial data and main results\label{sec Main results}}

Our construction of initial data is based upon \cite{RyNON21}. Consider the
potential (it is not the one we will deal with yet)%
\begin{equation}
q_{0}\left( x\right) =\left\{ 
\begin{array}{cc}
-2\dfrac{\mathrm{d}^{2}}{\mathrm{d}x^{2}}\log \left\{ 1+a\left( \sin
2x-2x\right) \right\} , & x<0 \\ 
0, & x\geq 0%
\end{array}%
\right. ,  \label{our Q}
\end{equation}%
where $a$ is a positive number. One can easily see that the function $q_{0}$
is continuous and $q_{0}\left( 0\right) =0$ but $q_{0}^{\prime }\left(
x\right) $ has a jump discontinuity at $x=0$. Moreover,%
\begin{equation}
q_{0}\left( x\right) =-4\ \dfrac{\sin 2x}{x}+O\left( \frac{1}{x^{2}}\right)
,\ \ x\rightarrow -\infty .  \label{Q asym}
\end{equation}%
Note that the leading term is independent of $a$. Apparently, $q_{0}\in
L^{2}\left( \mathbb{R}\right) $ but not even in $L^{1}\left( \mathbb{R}%
\right) $. Thus, $q_{0}$ is not short-range. Also note that%
\begin{equation*}
\int_{-\infty }^{\infty }q_{0}\left( x\right) \mathrm{d}x=0.
\end{equation*}%
The main feature of $q_{0}$ is that $\mathbb{L}_{q_{0}}$ admits an explicit
spectral and scattering theory. The Schr\"{o}dinger operator $\mathbb{L}%
_{q_{0}}$ on $L^{2}\left( \mathbb{R}\right) $ with $q_{0}$ given by (\ref%
{our Q}) has the following properties:

\begin{itemize}
\item The spectrum $\sigma \left( \mathbb{L}_{q_{0}}\right) =\left\{ -\kappa
^{2}\right\} \cup \sigma _{\mathrm{ac}}\left( \mathbb{L}_{q_{0}}\right) $
and $\sigma _{\mathrm{ac}}\left( \mathbb{L}_{q_{0}}\right) $ is two fold
purely absolutely purely continuous filling $[0,\infty )$. Here $\kappa >0$
solves the equation 
\begin{equation}
\kappa ^{3}+\kappa =2a.  \label{a}
\end{equation}

\item The left Jost solution $\psi \left( x,k\right) $ given by%
\begin{equation}
\psi \left( x,k\right) =\mathrm{e}^{-\mathrm{i}kx}\left\{ 1+\left( \frac{%
\mathrm{e}^{\mathrm{i}x}}{k-1}-\frac{\mathrm{e}^{-\mathrm{i}x}}{k+1}\right) 
\frac{2a\sin x}{1-2ax+a\sin 2x}\right\} ,\ \ \ x<0,  \label{left jost}
\end{equation}%
is a meromorphic function with two simple poles at $\pm 1$ (therefore $\pm 1$
are spectral singularities of order 1).

\item The (right) reflection coefficient $R^{0}$ and the transmission
coefficient $T^{0}$ are rational functions given by%
\begin{equation}
T^{0}\left( k\right) =\frac{k^{3}-k}{k^{3}-k+2\mathrm{i}a},\ \ R^{0}\left(
k\right) =\frac{-2\mathrm{i}a}{k^{3}-k+2\mathrm{i}a},
\end{equation}
\end{itemize}

It is convenient to remove the negative bound state $-\kappa ^{2}$ by
applying a single Darboux transformation (see e.g. \cite{Deift79}). This way
one gets a new potential%
\begin{equation*}
q\left( x\right) =q_{0}\left( x\right) -2\dfrac{\mathrm{d}^{2}}{\mathrm{d}%
x^{2}}\log \psi \left( x,\mathrm{i}\kappa \right) ,
\end{equation*}%
where $\psi $ is given by (\ref{left jost}). Since the single Darboux
transformation preserves support, $q\left( x\right) $ can be simplified to
read%
\begin{equation}
q\left( x\right) =\left\{ 
\begin{array}{cc}
-2\dfrac{\mathrm{d}^{2}}{\mathrm{d}x^{2}}\log \left\{ 1-\left( \kappa
^{3}+\kappa \right) x+\left( \kappa ^{3}-\kappa \right) \sin x\cos x+2\kappa
^{2}\sin ^{2}x\right\} , & x<0 \\ 
0, & x\geq 0%
\end{array}%
\right. ,  \label{our data}
\end{equation}%
where $\kappa $ is a positive number related to $a$ by $\left( \ref{a}%
\right) $. The scattering quantities transform as follows%
\begin{equation*}
T\left( k\right) =\frac{k-\mathrm{i}\kappa }{k+\mathrm{i}\kappa }T^{0}\left(
k\right) ,\ \ \ R\left( k\right) =-\frac{k-\mathrm{i}\kappa }{k+\mathrm{i}%
\kappa }R^{0}\left( k\right) .
\end{equation*}

For the reader's convenience we summarize important properties of $q$'s in
the following

\begin{theorem}
\label{Thm on Q}The Schr\"{o}dinger operator $\mathbb{L}_{q}$ on $%
L^{2}\left( \mathbb{R}\right) $ with $q$ given by (\ref{our Q}) has the
properties:

\begin{enumerate}
\item (Spectrum) $\sigma \left( \mathbb{L}_{q}\right) =\sigma _{\mathrm{ac}%
}\left( \mathbb{L}_{q}\right) $ $=[0,\infty )$ and multiplicity two. The
point $1$ is a first order WvN resonance.

\item (Scattering quantities) The transmission coefficient $T$ and the
(right) reflection coefficient $R$ are rational functions given by%
\begin{equation}
\begin{array}{ccccc}
T\left( k\right) & = & \dfrac{k-\mathrm{i}\kappa }{k+\mathrm{i}\kappa }%
\dfrac{k^{3}-k}{k^{3}-k+2\mathrm{i}a} & = & \dfrac{k^{3}-k}{\left( k+\mathrm{%
i}\kappa \right) \left( k-k_{-}\right) \left( k-k_{+}\right) } \\ 
R\left( k\right) & = & \dfrac{k-\mathrm{i}\kappa }{k+\mathrm{i}\kappa }%
\dfrac{2\mathrm{i}a}{k^{3}-k+2\mathrm{i}a} & = & \dfrac{\mathrm{i}\left(
\kappa ^{3}+\kappa \right) }{\left( k+\mathrm{i}\kappa \right) \left(
k-k_{-}\right) \left( k-k_{+}\right) }%
\end{array}%
,  \label{TR}
\end{equation}%
where%
\begin{equation*}
k_{\pm }=\pm \sqrt{1+3\kappa ^{2}/4}-\mathrm{i}\kappa /2.
\end{equation*}

\item (Asymptotic behavior) $q\left( x\right) $ has the asymptotic behavior%
\begin{equation}
q\left( x\right) =-4\ \dfrac{\sin\left( 2x+\delta\right) }{x}+O\left( \frac{1%
}{x^{2}}\right) ,\ \ x\rightarrow-\infty,  \label{q behavior}
\end{equation}
where%
\begin{equation*}
\tan\delta=\frac{2\kappa}{1-\kappa^{2}}.
\end{equation*}
\end{enumerate}
\end{theorem}

\begin{remark}
Observe that since the left $L$ and right $R$ reflection coefficient are
related by $L=-\left( T/\overline{T}\right) \overline{R}$, one readily sees
from (\ref{TR}) that $L=R$. Note the potential $q\left( -x\right) $ has the
same transition coefficients $\left( \ref{TR}\right) $ and hence $q\left(
x\right) $ and $q\left( -x\right) $ share the same scattering matrix meaning
that the set $S_{q}=\left\{ R\right\} $ does not make up scattering data.
This phenomenon was first observed in \cite{ADM81} for certain potentials
decaying as $O\left( x^{-2}\right) $. We mention that the construction in 
\cite{ADM81} involves a delta-function whereas our counterexample is given
by a continuous function. Note that the uniqueness is restored once we a
priori assume that $q\left( x\right) $ is supported on $\left( -\infty
,0\right) $.
\end{remark}

Note that $\gamma=1$ for any $q$ of the form (\ref{our data}) and thus the
results of \cite{RyNON21} do not apply as we need $\gamma<1/2$. However
since the support of each $q$ is restricted to $\left( -\infty,0\right) $
the IST does apply \cite{RybNON2010} and, as in the classical IST, the time
evolved scattering data for the whole class is given by%
\begin{equation}
S_{q}(t)=\left\{ R(k)\exp\left( 8\mathrm{i}k^{3}t\right) \right\} .
\label{time evol}
\end{equation}
The general theory \cite{RybNON2010} implies that $q\left( x,t\right) $ is
meromorphic in $x$ on the whole complex plane for any $t>0$. Moreover, for
each fixed $t>0$, $q\left( x,t\right) $ decays exponentially as $%
x\rightarrow+\infty$ and behaves as $\left( \ref{q behavior}\right) $ at $%
-\infty$ \cite{NovikovKhenkin84}. We now state our main result.

\begin{theorem}[On resonance asymptotic regime]
\label{Thm on WvN} Let $q(x)$ be of the form (\ref{our data}) with some
positive $\kappa $. Then the solution $q\left( x,t\right) $ to the KdV
equation with the initial data $q\left( x\right) $ behaves along the line $%
x=-12t$ as follows%
\begin{eqnarray}
\left. q\left( x,t\right) \right\vert _{x=-12t} &\sim &\sqrt{\frac{4\nu
\left( t\right) }{3t}}\sin \left\{ 16t-\nu \left( t\right) \left[ 1+\pi \nu
\left( t\right) -\log \nu \left( t\right) \right] +\delta \right\} ,  \notag
\\
t &\rightarrow &+\infty ,  \label{eq res asym}
\end{eqnarray}%
where%
\begin{equation*}
\nu \left( t\right) =\left( 1/2\pi \right) \log \left( 12\left( \kappa
^{3}+\kappa \right) ^{2}t\right) ,
\end{equation*}%
and a constant phase $\delta $ is given by%
\begin{align*}
\delta & =\frac{2}{\pi }\int\limits_{0}^{1}\log \left\vert \frac{%
T(s)/T^{\prime }\left( 1\right) }{1-s}\right\vert ^{2}\frac{\mathrm{d}s}{%
1-s^{2}}-\frac{2}{\pi }\int\limits_{0}^{1}\frac{\log s{\mathrm{d}s}}{s-2}. \\
& +\frac{1}{\pi }\log \frac{\kappa ^{3}+\kappa }{2}\log \frac{\kappa
^{3}+\kappa }{8}+\arctan \frac{2\kappa }{1-\kappa ^{2}}-\frac{\pi }{3}.
\end{align*}
\end{theorem}

By way of discussing this theorem we offer some comments.

\begin{remark}
The class of initial data of the form $\left( \ref{our data}\right) $
appears extraordinarily narrow. But of course by a simple rescaling we can
place the resonance at any point $\omega ^{2}$. Secondly, $\left( \ref{our
data}\right) $ shares the same asymptotic behavior as any WvN type potential
with $\gamma =1$ restricted to $\left( -\infty ,0\right) $ but have explicit
rational scattering data, which has significant technical advantages.
\end{remark}

\begin{remark}
It is shown in \cite{NovikovKhenkin84} that simple zeros of the transmission
coefficient are always associated with the behavior $\left( \ref{pure WvN}%
\right) $ with $\gamma =1$. The approach is based upon the
Gelfand-Levitan-Marchenko equation and does not yield the existence of the
resonance regime. Incidentally, the very concept of WvN resonance (nor
spectral singularity to this matter) is not used in \cite{NovikovKhenkin84}.
\end{remark}

\begin{remark}
Recall that in the short-range case the behavior of $\left. q\left(
x,t\right) \right\vert _{x=-ct}$ as $t\rightarrow +\infty $ looks similar to 
$\left( \ref{eq res asym}\right) $ for any $c>0$ (the \emph{similarity region%
}) with one main difference: $\nu $ is constant. Thus, as opposed to the
classical case where any radiant wave decays as $O\left( t^{-1/2}\right) $
when $t\rightarrow \infty $, our initial date (\ref{our data}) give rise to
a radiant wave that decays as $O\left( \left( \log t/t\right) ^{1/2}\right) $%
. We believe that it is a new phenomenon that holds for any initial data
producing zero transmission (full reflection) at some isolated positive
energies.
\end{remark}

\begin{remark}
In the NLS context, asymptotic regions related to spectral singularities are
studied in \cite{AS77},\cite{Budylin20},\cite{Kamvissis94} by different
methods. It follows from \cite{AS77} that for the focusing NLS a spectral
singularity opens a new asymptotic region where the amplitude decays as $%
O\left( \left( \log t/t\right) ^{1/2}\right) $ and the phase includes $\log
^{2}t$ term (similar to our resonance regime). The part of the paper \cite%
{AS77} devoted to the NLS has no proofs and is revisited in \cite%
{Kamvissis94} where it is claimed that the amplitude in \cite{AS77} is off
by a multiplicative constant and the phase has no $\log ^{2}t$ term. The
de-focusing case is considered in \cite{Budylin20} where the results also
suggest the appearance of a new asymptotic region with the amplitude
behaving like $O\left( \left( \log t/t\right) ^{1/2}\right) $ and the phase
containing a $\log ^{2}t$ term. The paper \cite{Budylin20} relies on totally
different arguments and the author appears to be unaware of \cite{AS77} and 
\cite{Kamvissis94}. It appears that there is no clarity on the issue and we
hope to address it elsewhere.
\end{remark}

\begin{remark}
In the short-range case, generically , $T\left( 0\right) =0$ (i.e. $T\left(
0\right) \neq 0$ only for some exceptional potentials). As it is shown in 
\cite{AS77},\cite{DVZ94}, the latter gives rise to a new asymptotic regime,
called the \emph{collisionless shock region}. We referee the interested
reader to \cite{AS77},\cite{DVZ94} and only mention here that $\left(
x,t\right) $ in this region are no longer on a ray and the wave amplitude
decays as $O\left( \left( \log t/t\right) ^{2/3}\right) $ as $t\rightarrow
+\infty .$
\end{remark}

\section{Original RHP\label{sec original RHP}}

As is always done in the RHP approach to the IST, we rewrite the (time
evolved)\ basic scattering identity (\ref{R basic scatt identity}) as a
vector RHP. To this end, introduce a $2\times 1$ row vector%
\begin{equation}
\mathbf{m}(k,x,t)=\left\{ 
\begin{array}{c@{\quad}l}
\begin{pmatrix}
T(k)m^{\left( -\right) }(k,x,t) & m^{\left( +\right) }(k,x,t)%
\end{pmatrix}%
, & \func{Im}k>0, \\ 
\begin{pmatrix}
m^{\left( +\right) }(-k,x,t) & T(-k)m^{\left( -\right) }(-k,x,t)%
\end{pmatrix}%
, & \func{Im}k<0,%
\end{array}%
\right.  \label{defm}
\end{equation}%
where%
\begin{equation}
m^{\left( \pm \right) }(k;x,t):=\mathrm{e}^{\mp \mathrm{i}kx}\psi ^{\left(
\pm \right) }(x,t;k),\text{\ \ \ (Faddeev functions).}  \label{y}
\end{equation}%
Note that we list $k$ as the first variable as from now on it will be the
main variable, the parameters $\left( x,t\right) $ being often suppressed.
Existence of $\psi ^{\left( \pm \right) }(x,t;k)$ for $t>0$ follows from
considerations given in \cite{NovikovKhenkin84}. Also, since poles of $\psi
^{\left( -\right) }$ are canceled by the zeros of $T$, we conclude that $%
\mathbf{m}$ is bounded on the whole complex plane. Details will be provided
elsewhere for a much wider class of initial data.

Introduce the short hand notation%
\begin{equation*}
\xi\left( k\right) =\exp\left\{ \mathrm{i}t\Phi\left( k\right) \right\} ,\
\Phi\left( k\right) =4k^{3}+\left( x/t\right) k,\ \ x\in\mathbb{R},t\geq0.
\end{equation*}
We treat $\mathbf{m}$ as a solution to

\textbf{RHP\#0. (Original RHP) }Find a $1\times2$ matrix (row) valued
function $\mathbf{m}\left( k\right) $ which is analytic and bounded on $%
\mathbb{C\diagdown R}$ and satisfies:

(1)\ The jump condition%
\begin{equation}
\mathbf{m}_{+}\left( k\right) =\mathbf{m}_{-}\left( k\right) \mathbf{V}%
\left( k\right) \text{ for }k\in \mathbb{R},  \label{main RHP}
\end{equation}%
where $\mathbf{m}_{\pm }\left( k\right) :=\mathbf{m}\left( k\pm \mathrm{i}%
0\right) ,\ k\in \mathbb{R}$,%
\begin{equation}
\mathbf{V}\left( k\right) =\left( 
\begin{array}{cc}
1-\left\vert R\left( k\right) \right\vert ^{2} & -\overline{R}\left(
k\right) \xi \left( k\right) ^{-2} \\ 
R\left( k\right) \xi \left( k\right) ^{2} & 1%
\end{array}%
\right) ,\text{ (jump matrix),}  \label{jump matrix}
\end{equation}%
and $R$ is given by $\left( \ref{TR}\right) $;

(2) The symmetry condition%
\begin{equation*}
\mathbf{m}\left( -k\right) =\mathbf{m}\left( k\right) \sigma_{1}\text{;}
\end{equation*}

(3)\ The normalization condition%
\begin{equation*}
\mathbf{m}\left( k\right) \sim \left( 
\begin{array}{cc}
1 & 1%
\end{array}%
\right) ,\ \ \ k\rightarrow \infty .
\end{equation*}

The solution to the initial problem $\left( \ref{KdV}\right) -\left( \ref%
{KdVID}\right) $ can then be found by%
\begin{equation}
q\left( x,t\right) =2\lim k^{2}\left( m_{1}\left( k,x,t\right) m_{2}\left(
k,x,t\right) -1\right) ,\ \ \ k\rightarrow\infty  \label{q(x,t)}
\end{equation}
where $m_{1,2}$ are the entries of $\mathbf{m}$ (see e.g. \cite{GT09}).

The RHP approach is specifically robust (among others) in asymptotics
analysis of $\mathbf{m}\left( k;x,t\right) $ as $t\rightarrow +\infty $ in
different asymptotic regions \cite{GT09}. But it works well if the
stationary point $k_{0}=\sqrt{-x/12t}$ of $\Phi \left( k\right) $ is not a
real zero of $T\left( k\right) $. Recall that in the short-range case the
only real zero of $T\left( k\right) $ is $k=0$ and does not create any
problems as long as $k_{0}\neq 0$. As already mentioned, the case of $%
k_{0}=0 $ is considered in \cite{DVZ94}. Note however that $0$ is the left
end point of the absolutely continuous spectrum, which is essential for the
approach to work.

\begin{remark}
There is no problem with the well-posedness of RHP\#0 and no problem with
the deformation step as $\mathbf{V}$ is a meromorphic function on the entire
plane. But there is a problem with adjusting the classical nonlinear
steepest descent \cite{DeiftZhou1993}. Recall that in the mKdV case treated
in \cite{DeiftZhou1993} we always have $\left\vert R(k)\right\vert <1$ and
hence $R(k)/(1-|R(k)|^{2})$ can be approximated by analytic functions in the 
$L^{\infty}$ norm. As we have seen already, it is not our case and it
appears to be a good open question how to adjust the nonlinear steepest
descent to RHP\#0.
\end{remark}

\begin{remark}
RHP\#0 does not appear singular so far. Its singularity will be clear below.
\end{remark}

\section{Partial Transmission Coefficient\label{sec Partial T}}

In this section we study what is commonly referred to as the\emph{\ partial
transmission coefficient}, which plays the crucial role in the conjugation
step discussed below. Following \cite{GT09} we introduce it by 
\begin{equation*}
T(k,k_{0})=\exp\int\limits_{\left( -k_{0},k_{0}\right) }\frac{\log |T(s)|^{2}%
}{s-k}\frac{{\mathrm{d}s}}{2\pi\mathrm{i}}.
\end{equation*}
Since we are only concerned with the asymptotics related to the resonance
region we take $k_{0}=1$ and denote%
\begin{equation}
T_{0}\left( k\right) =T(k,1)=\exp\int\limits_{\left( -1,1\right) }\frac{%
\log|T(s)|^{2}}{s-k}\frac{{\mathrm{d}s}}{2\pi\mathrm{i}}.  \label{T_0}
\end{equation}
If $k_{0}=+\infty$ then $T\left( k,+\infty\right) $ returns the (full)
transmission coefficient $T\left( k\right) $.

$T_{0}(k)$ is clearly an analytic function for $k\in \mathbb{C}\backslash %
\left[ -1,1\right] $. The following statement extends that of \cite{GT09}.

\begin{theorem}
\label{thm:part} The partial transmission coefficient $T_{0}(k)$ satisfies
the jump condition 
\begin{equation}
T_{0+}(k)=T_{0-}(k)(1-\left\vert R(k)\right\vert ^{2}),\ \ k\in \left(
-1,1\right)   \label{eq:jumpt}
\end{equation}%
and

\begin{enumerate}
\item $T_{0}(-k)=T_{0}(k)^{-1}$, $T_{0}\left( \overline{k}\right) =\overline{%
T_{0}\left( k\right) }^{-1},k\in\mathbb{C}\backslash\left[ -1,1\right] $;

\item $T_{0}(-k)=\overline{T_{0}(\overline{k})}$, $k\in\mathbb{C}$, in
particular $T_{0}(k)$ is real for $k\in$\textrm{$i$}$\mathbb{R}$;

\item the behavior near $k=0$ is given by $T_{0}(k)\sim CT(k)$ with $C\neq0$
for $\func{Im}k\geq0$;

\item $T_{0}\left( k\right) $ can be factored out as%
\begin{equation}
T_{0}\left( k\right) =T_{0}^{\left( \mathrm{reg}\right) }\left( k\right)
T_{0}^{\left( \mathrm{sng}\right) }\left( k\right) ,  \label{T_0 at 1}
\end{equation}%
where the regular part%
\begin{equation*}
T_{0}^{\left( \mathrm{reg}\right) }\left( k\right) =\frac{{k}}{\pi \mathrm{i}%
}\int\limits_{0}^{1}\log \left\vert \frac{T(s)/\left( 1-s\right) }{T^{\prime
}\left( 1\right) }\right\vert ^{2}\frac{\mathrm{d}s}{s^{2}-k^{2}}
\end{equation*}%
is continuously differentiable at $k=\pm 1$, and the singular part%
\begin{equation}
T_{0}^{\left( \mathrm{sng}\right) }\left( k\right) =\left( \frac{k-1}{k+1}%
\right) ^{\mathrm{i}\nu }\exp \left\{ k\int_{0}^{1}\frac{\log \left(
1-s\right) ^{2}}{s^{2}-k^{2}}\frac{\mathrm{d}s}{\pi \mathrm{i}}\right\} ,
\label{sing part}
\end{equation}%
\begin{equation}
\nu :=-\left( 1/\pi \right) \log \left\vert T^{\prime }\left( 1\right)
\right\vert =\left( 1/\pi \right) \log a.  \label{nu}
\end{equation}%
Moreover,%
\begin{equation}
T_{0}\left( k\right) \sim A_{0}\exp \left\{ \mathrm{i}\nu \log \left(
k-1\right) +\frac{1}{2\mathrm{i}\pi }\log ^{2}\left( k-1\right) \right\}
,k\rightarrow 1,  \label{T_0 asympt}
\end{equation}%
where $A_{0}$ is a unimodular constant given by%
\begin{equation}
A_{0}=\exp \frac{\mathrm{i}}{\pi }\left\{ \int\limits_{0}^{1}\log \left\vert 
\frac{T(s)/T^{\prime }\left( 1\right) }{1-s}\right\vert ^{2}\frac{\mathrm{d}s%
}{1-s^{2}}-\pi \nu \log 2-\frac{\pi ^{2}}{6}-\int\limits_{0}^{1}\frac{\log s{%
\mathrm{d}s}}{s-2}\right\} .  \label{A_0}
\end{equation}
\end{enumerate}
\end{theorem}

\begin{proof}
The items (1)-(3) are straightforward to check. Item (4) is a bit technical.
Since $\left\vert T\left( s\right) \right\vert $ is even, we clearly have%
\begin{align*}
\int\limits_{-1}^{1}\frac{\log |T(s)|^{2}}{s-k}{\mathrm{d}s}& {=2k}%
\int\limits_{0}^{1}\frac{\log |T(s)|^{2}}{s^{2}-k^{2}}{\mathrm{d}s} \\
& ={2k}\int\limits_{0}^{1}\frac{\log |T(s)/\left( 1-s\right) |^{2}}{%
s^{2}-k^{2}}{\mathrm{d}s+2k}\int\limits_{0}^{1}\frac{\log \left( 1-s\right)
^{2}}{s^{2}-k^{2}}{\mathrm{d}s}
\end{align*}%
and therefore%
\begin{equation*}
T_{0}\left( k\right) =\exp \left\{ k\int\limits_{0}^{1}\frac{\log
|T(s)/\left( 1-s\right) )|^{2}}{s^{2}-k^{2}}\frac{{\mathrm{d}s}}{\pi \mathrm{%
i}}\right\} \exp \left\{ k\int\limits_{0}^{1}\frac{\log \left( 1-s\right)
^{2}}{s^{2}-k^{2}}\frac{{\mathrm{d}s}}{\pi \mathrm{i}}\right\} .
\end{equation*}%
It remains to factor out the continuous part of the first factor on the
right hand side. To this end note that%
\begin{equation*}
\lim_{s\rightarrow 1}|T(s)/\left( 1-s\right) )|=\left\vert
\lim_{s\rightarrow 1}T(s)/\left( 1-s\right) \right\vert =\left\vert
T^{\prime }\left( 1\right) \right\vert =1/a
\end{equation*}%
and therefore%
\begin{align*}
& k\int\limits_{0}^{1}\frac{\log |T(s)/\left( 1-s\right) )|^{2}}{s^{2}-k^{2}}%
\frac{{\mathrm{d}s}}{\pi \mathrm{i}} \\
& =\int\limits_{0}^{1}\frac{k}{s^{2}-k^{2}}\log \left\vert \frac{T(s)/\left(
1-s\right) )}{T^{\prime }\left( 1\right) }\right\vert ^{2}\frac{{\mathrm{d}s}%
}{\pi \mathrm{i}}+\log \left\vert T^{\prime }\left( 1\right) \right\vert
^{2}\int\limits_{0}^{1}\frac{k}{s^{2}-k^{2}}\frac{{\mathrm{d}s}}{\pi \mathrm{%
i}}.
\end{align*}%
Eq $\left( \ref{sing part}\right) $ follows now once we observe that%
\begin{equation*}
\nu \int\limits_{0}^{1}\frac{k}{s^{2}-k^{2}}\frac{{\mathrm{d}s}}{\pi \mathrm{%
i}}=\left( \frac{k-1}{k+1}\right) ^{\mathrm{i}\nu },
\end{equation*}%
where $\nu $ is given by $\left( \ref{nu}\right) $. It remains to
demonstrate $\left( \ref{T_0 asympt}\right) $. Consider%
\begin{equation*}
k\int\limits_{0}^{1}\frac{\log \left( 1-s\right) ^{2}}{s^{2}-k^{2}}{\mathrm{d%
}s}=\int\limits_{0}^{1}\frac{\log \left( 1-s\right) }{s-k}{\mathrm{d}s+}%
\int\limits_{-1}^{0}\frac{\log \left( 1+s\right) }{s-k}{\mathrm{d}s=:I}%
_{1}\left( k\right) +I_{2}\left( k\right) .
\end{equation*}%
By a direct computation, for ${I}_{1}\left( k\right) $ we have%
\begin{equation*}
I_{1}\left( k\right) =\int\limits_{0}^{1}\frac{\log \left( 1-s\right) }{s-k}{%
\mathrm{d}s\sim }\frac{1}{2}\log ^{2}\left( k-1\right) +\frac{\pi ^{2}}{6},\
\ \ k\rightarrow 1.
\end{equation*}%
Since $I_{2}\left( k\right) $ is clearly continuously differentiable at $k=1$
one has%
\begin{align*}
I_{2}\left( k\right) & \sim I_{2}\left( 1\right) =\int\limits_{-1}^{0}\frac{%
\log \left( 1+s\right) }{s-1}{\mathrm{d}s} \\
& {=}\int\limits_{0}^{1}\frac{\log \lambda }{\lambda -2}{\mathrm{d}\lambda }%
,\ \ \ k\rightarrow 1
\end{align*}%
and thus%
\begin{align*}
& k\int\limits_{0}^{1}\frac{\log \left( 1-s\right) ^{2}}{s^{2}-k^{2}}{%
\mathrm{d}s} \\
& {\sim }\frac{1}{2}\log ^{2}\left( k-1\right) +\frac{\pi ^{2}}{6}%
+\int\limits_{0}^{1}\frac{\log \lambda }{\lambda -2}{\mathrm{d}\lambda },\ \
\ k\rightarrow 1.
\end{align*}%
We have%
\begin{align*}
& \exp \left\{ k\int\limits_{0}^{1}\frac{\log \left( 1-s\right) ^{2}}{%
s^{2}-k^{2}}\frac{{\mathrm{d}s}}{\pi \mathrm{i}}\right\} \\
& \sim \exp \left\{ \frac{1}{2\pi \mathrm{i}}\log ^{2}\left( k-1\right)
\right\} \exp \left\{ -\frac{\mathrm{i}\pi }{6}-\frac{\mathrm{i}}{\pi }%
\int\limits_{0}^{1}\frac{\log \lambda {\mathrm{d}\lambda }}{\lambda -2}%
\right\} ,\ \ \ k\rightarrow 1.
\end{align*}
\end{proof}

\begin{remark}
In the classical case the singular part of $T_{0}\left( k\right) $ is%
\begin{equation*}
T_{0}^{\left( \mathrm{sng}\right) }\left( k\right) =\left( \frac{k-1}{k+1}%
\right) ^{\mathrm{i}\nu }
\end{equation*}%
which is a bounded function. The extra factor in $\left( \ref{sing part}%
\right) $ containing $\log ^{2}$ term leads to a more singular (unbounded)
behavior of $T_{0}\left( k\right) $ at $k=\pm 1$. The presence of the $\log
^{2}$ term is responsible for the new asymptotic regime.
\end{remark}

\section{First Conjugation and Deformation\label{sec:condef}}

This section demonstrates how to conjugate our RHP\#0 and how to deform our
jump contour, such that the jumps will be exponentially close to the
identity away from the stationary phase points. This step is the same as in
the classical case \cite{GT09} and we just directly record the result:
RHP\#1 below is equivalent to RHP\#0.

\textbf{RHP\#1 (Conjugated RHP\#0) }Find a $1\times2$ matrix (row) valued
function $\mathbf{m}^{\left( 1\right) }\left( k\right) $ which is analytic
on $\mathbb{C\diagdown R}$ and satisfies:

(1)\ The jump condition%
\begin{equation}
\mathbf{m}_{+}^{\left( 1\right) }\left( k\right) =\mathbf{m}_{-}^{\left(
1\right) }\left( k\right) \mathbf{V}^{\left( 1\right) }\left( k\right) \text{
for }k\in\mathbb{R},
\end{equation}
where%
\begin{equation}
\mathbf{V}^{\left( 1\right) }\left( k\right) =\mathbf{B}_{-}(k)^{-1}\mathbf{B%
}_{+}(k),
\end{equation}%
\begin{equation}
\mathbf{B}_{-}(k)=\left\{ 
\begin{array}{ccc}
\begin{pmatrix}
1 & \dfrac{R(-k)T_{0}\left( k\right) ^{2}}{\xi\left( k\right) ^{2}} \\ 
0 & 1%
\end{pmatrix}
& , & k\in\mathbb{R}\backslash\left[ -1,1\right] \\ 
\begin{pmatrix}
1 & 0 \\ 
-\dfrac{R(k)\xi\left( k\right) ^{2}}{T_{0-}(k)^{2}T\left( k\right) T\left(
-k\right) } & 1%
\end{pmatrix}
& , & k\in\left[ -1,1\right]%
\end{array}
\right. ,  \label{B -}
\end{equation}%
\begin{equation}
\mathbf{B}_{+}(k)=\left\{ 
\begin{array}{ccc}
\begin{pmatrix}
1 & 0 \\ 
\dfrac{R(k)\xi\left( k\right) ^{2}}{T_{0}\left( k\right) ^{2}} & 1%
\end{pmatrix}
& , & k\in\mathbb{R}\backslash\left[ -1,1\right] \\ 
\begin{pmatrix}
1 & -\dfrac{T_{0+}\left( k\right) ^{2}R(-k)}{T\left( k\right) T\left(
-k\right) \xi\left( k\right) ^{2}} \\ 
0 & 1%
\end{pmatrix}
& , & k\in\left[ -1,1\right]%
\end{array}
\right. ;  \label{B+}
\end{equation}

(2) The symmetry condition%
\begin{equation*}
\mathbf{m}^{\left( 1\right) }\left( -k\right) =\mathbf{m}^{\left( 1\right)
}\left( k\right) \sigma_{1}\text{;}
\end{equation*}

(3)\ The normalization condition%
\begin{equation*}
\mathbf{m}^{\left( 1\right) }\left( k\right) \sim \left( 
\begin{array}{cc}
1 & 1%
\end{array}%
\right) ,k\rightarrow \infty .
\end{equation*}

The solutions to RHP\#0 and RHP\#1 are related by%
\begin{equation}
\mathbf{m}\left( k\right) =\mathbf{m}^{\left( 1\right) }\left( k\right)
T_{0}\left( k\right) ^{\sigma _{3}}.  \label{eq m and m^1}
\end{equation}

We show that $\mathbf{B}_{\pm }$ admits analytic continuation into $\mathbb{C%
}_{\pm }$ respectively with no poles. Consider $\mathbf{B}_{+}\left(
k\right) $ only. Due to $\left( \ref{TR}\right) $ we have%
\begin{equation*}
\dfrac{R(-k)}{T\left( k\right) T\left( -k\right) }=-2\mathrm{i}a\dfrac{%
\left( k+\mathrm{i}\kappa \right) \left( k-k_{-}\right) \left(
k-k_{+}\right) }{\left( k^{3}-k\right) ^{2}}
\end{equation*}%
and thus the matrix entry $-\dfrac{T_{0+}\left( k\right) ^{2}R(-k)}{T\left(
k\right) T\left( -k\right) \xi \left( k\right) ^{2}}$ can be continued from $%
\left[ -1,1\right] $ to $\mathbb{C}^{+}$ (with no poles). Similarly, one
easily sees that $\dfrac{R(k)\xi \left( k\right) ^{2}}{T_{0}\left( k\right)
^{2}}$ continues from $\mathbb{R}\backslash \left[ -1,1\right] $ to $\mathbb{%
C}^{+}$. Finally notice that due to item 3 of Theorem \ref{thm:part}, $k=0$
is not a singularity of $\mathbf{B}_{\pm }\left( k\right) $ and therefore,
RHP\#1 can be deformed the same way as it is done in \cite{GT09}. Denote by $%
\Gamma $ the (oriented) contour depicted in Figure 1 and by $\Gamma _{\pm
}=\Gamma \cap \mathbb{C}^{\pm }$ (parts of $\Gamma $ in the upper/lower
half-planes) chosen such that $\Gamma _{-}$ and $\Gamma _{+}$ lie in the
region where the irrespective power $\pm 1$ of $\left\vert \xi \left(
k\right) \right\vert =\exp \left( -t\func{Im}k\right) $ provides an
exponential decay as $t\rightarrow +\infty $.

\textbf{RHP\#2 (Deformed RHP\#1)}. Find a $1\times2$ matrix (row) valued
function $\mathbf{m}^{\left( 2\right) }\left( k\right) $ which is analytic
on $\mathbb{C\diagdown}\Gamma$ and satisfies:

(1)\ The jump condition%
\begin{equation}
\mathbf{m}_{+}^{\left( 2\right) }\left( k\right) =\mathbf{m}_{-}^{\left(
2\right) }\left( k\right) \mathbf{V}^{\left( 2\right) }\left( k\right) \text{
for }k\in \Gamma ,
\end{equation}%
where 
\begin{equation}
\mathbf{V}^{\left( 2\right) }(k)=%
\begin{cases}
\mathbf{B}_{+}(k), & k\in \Gamma _{+} \\ 
\mathbf{B}_{-}(k)^{-1}, & k\in \Gamma _{-}%
\end{cases}%
;
\end{equation}

(2) The symmetry condition%
\begin{equation*}
\mathbf{m}^{\left( 2\right) }\left( -k\right) =\mathbf{m}^{\left( 2\right)
}\left( k\right) \sigma_{1}\text{;}
\end{equation*}

(3)\ The normalization condition%
\begin{equation*}
\mathbf{m}^{\left( 2\right) }\left( k\right) \sim \left( 
\begin{array}{cc}
1 & 1%
\end{array}%
\right) ,k\rightarrow \infty .
\end{equation*}

The solutions of RHP\#2 and RHP\#1 are related by%
\begin{equation}
\mathbf{m}^{\left( 2\right) }(k)=%
\begin{cases}
\mathbf{m}^{\left( 1\right) }(k)\mathbf{B}_{+}(k)^{-1}, & k\text{ is between 
}\mathbb{R\ }\text{and }\Gamma _{+} \\ 
\mathbf{m}^{\left( 1\right) }(k)\mathbf{B}_{-}(k)^{-1}, & k\text{ is between 
}\mathbb{R\ }\text{and }\Gamma _{-} \\ 
\mathbf{m}^{\left( 1\right) }(k), & \text{elsewhere}%
\end{cases}%
.
\end{equation}

In fact, the jump along $\mathbb{R}$ disappears but the solution between $%
\Gamma_{-}$ and $\Gamma_{+}$ will not be needed, only outside where 
\begin{equation}
\mathbf{m}^{\left( 2\right) }(k)=\mathbf{m}^{\left( 1\right) }(k).
\label{RHP1 = RHP2}
\end{equation}

The main feature of RHP\#2 is that off-diagonal entries of $\mathbf{V}%
^{\left( 2\right) }$ along $\Gamma \backslash \{-1,1\}$ are exponentially
decreasing as $t\rightarrow +\infty $. 
\begin{figure}[tbp]
\centering
\begin{picture}(7,5.2)
		\put(0,2.5){\line(1,0){7.0}}

		\put(7.3,2.4){$\R$}
		
		\put(0,2){\line(1,0){2.0}}
		\put(2.9,3){\line(1,0){1.2}}
		\put(5,2){\line(1,0){2.0}}
		\put(1.1,2){\vector(1,0){0.4}}
		\put(5.5,2){\vector(1,0){0.4}}
		\put(3.3,3){\vector(1,0){0.4}}
		
		\put(1.3,1.5){$\Gamma_-$}
		\put(5.7,1.5){$\Gamma_-$}
		\put(3.5,3.3){$\Gamma_+$}
		
		\curve(2.,2., 2.2,2.1, 2.4,2.4, 2.45,2.5, 2.5,2.6, 2.7,2.9, 2.9,3.)
		
		\curve(4.1,3., 4.3,2.9, 4.5,2.6, 4.55,2.5, 4.6,2.4, 4.8,2.1, 5.,2.)
		
		\put(0,3){\line(1,0){2.0}}
		\put(2.9,2){\line(1,0){1.2}}
		\put(5,3){\line(1,0){2.0}}
		\put(1.1,3){\vector(1,0){0.4}}
		\put(5.5,3){\vector(1,0){0.4}}
		\put(3.3,2){\vector(1,0){0.4}}
		
		\curve(2.,3., 2.2,2.9, 2.4,2.6, 2.45,2.5, 2.5,2.4, 2.7,2.1, 2.9,2.)
		
		\curve(4.1,2., 4.3,2.1, 4.5,2.4, 4.55,2.5, 4.6,2.6, 4.8,2.9, 5.,3.)
		
		\put(1.3,3.3){$\Gamma_+$}
		\put(5.7,3,3){$\Gamma_+$}
		\put(3.5,1.5){$\Gamma_-$}
		
		\put(0.3,1.0){$\scriptstyle\im\Phi<0$}
		\put(0.3,3.8){$\scriptstyle\im\Phi>0$}
		\put(5.9,1.0){$\scriptstyle\im\Phi<0$}
		\put(5.9,4.0){$\scriptstyle\im\Phi>0$}
		\put(2.9,4.5){$\scriptstyle\im\Phi<0$}
		\put(2.9,0.5){$\scriptstyle\im\Phi>0$}
		
		\put(2.6,2.6){$\scriptstyle -1$}
		\put(4.7,2.6){$\scriptstyle 1$}
		
		\curvedashes{0.05,0.05}
		
		\curve(0.,0.05, 0.85,0.5, 1.55,1., 2.05,1.5, 2.45,2.5, 2.05,3.5, 1.55,4., 0.85,4.5, 0.,4.94)
		
		\curve(7.,0.05, 6.15,0.5, 5.45,1., 4.95,1.5, 4.55,2.5, 4.95,3.5, 5.45,4., 6.15,4.5, 7.,4.94)
	\end{picture}
\caption{Signature plane and deformed contour}
\end{figure}
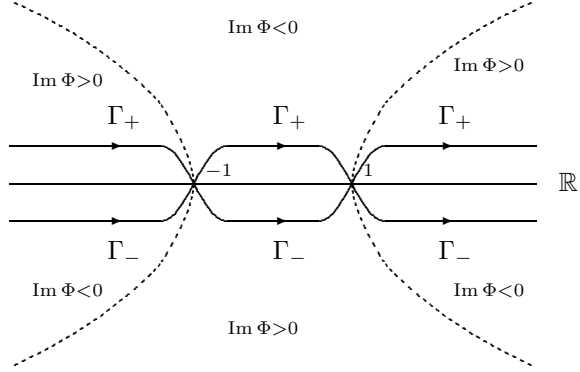

\begin{remark}
RHP\#1 and hence RHP\#2 are singular. Indeed, $\left( \ref{eq m and m^1}%
\right) $ implies that the first entry of $\mathbf{m}^{\left( 1\right) }(k)$
is unbounded at $k=\pm 1.$ Since the transformation $\left( \ref{eq m and
m^1}\right) $ is just algebraic, well-posedness is not an issue.
\end{remark}

\section{Matrix RHP\ on crosses\label{sec Matrix RHP}}

In the previous section we have reduced our original RHP to a RHP for $%
\mathbf{m}^{\left( 2\right) }(k)$ such that away from $\pm 1$ the jumps are
exponentially close to the identity as $t\rightarrow +\infty $ and only
parts of $\Gamma $ that are close to the stationary points $\pm 1$
contribute to the solution. This prompts a new RHP, on two small crosses $%
\Gamma _{\pm 1}=\cup _{j=1}^{4}\gamma _{\pm j}$ where $\gamma _{\pm j}$ are
given ($\epsilon $ is a small enough number which value is not essential)%
\begin{align*}
\gamma _{\pm 1}& =\{\pm 1+u\mathrm{e}^{-\mathrm{i}\pi /4},\,u\in \left[
0,\epsilon \right] \} & \gamma _{\pm 2}& =\{\pm 1+u\mathrm{e}^{\mathrm{i}\pi
/4},\,u\in \left[ 0,\epsilon \right] \} \\
\gamma _{\pm 3}& =\{\pm 1+u\mathrm{e}^{3\mathrm{i}\pi /4},\,u\in \left[
0,\epsilon \right] \} & \gamma _{\pm 4}& =\{\pm 1+u\mathrm{e}^{-3\mathrm{i}%
\pi /4},\,u\in \left[ 0,\epsilon \right] \},
\end{align*}%
oriented in the direction of the increase of $\func{Re}k$.

As typical in such a situation, we state our new RHP as a matrix one:

\textbf{RHP\#3 (Matrix RHP on small crosses)}. Find a $2\times 2$ matrix
valued function $\mathbf{M}^{\left( 3\right) }\left( k\right) $ which is
analytic away from $\Gamma _{-1}\cup \Gamma _{+1}$ and satisfies:

(1)\ The jump condition%
\begin{equation*}
\mathbf{M}_{+}^{\left( 3\right) }\left( k\right) =\mathbf{M}_{-}^{\left(
3\right) }\left( k\right) \mathbf{V}^{\left( 3\right) }\left( k\right) \text{
for }k\in \Gamma _{-1}\cup \Gamma _{+1},
\end{equation*}%
where%
\begin{equation*}
\mathbf{V}^{\left( 3\right) }(k)=\mathbf{V}^{\left( 2\right) }\left(
k\right) ,\ \ \ k\in \gamma _{\pm j},\ \ \ j=1,2,3,4;
\end{equation*}

(2) The symmetry condition%
\begin{equation}
\mathbf{M}^{\left( 3\right) }\left( -k\right) =\sigma_{1}\mathbf{M}^{\left(
3\right) }\left( k\right) \sigma_{1}\text{;}  \label{symmetry cond for mat}
\end{equation}

(3)\ The normalization condition%
\begin{equation*}
\mathbf{M}^{\left( 3\right) }\left( k\right) \sim \mathbf{I},\ \ \
k\rightarrow \infty .
\end{equation*}

Note that our enumeration of $\gamma _{j}$ is chosen to agree with that of 
\cite{GT09}. We now use the following result (see e.g. \cite{GT09}). Suppose
that for every sufficiently small $\epsilon >0$ both the $L^{2}$ and the $%
L^{\infty }$ norms of $\mathbf{V}^{\left( 2\right) }-\mathbf{I}$ are $%
O(b\left( t\right) ^{-\beta })$ with some $b\left( t\right) \rightarrow
+\infty $ (given below) and $\beta >0$ away from some $\varepsilon $
neighborhoods of the points $k=\pm 1$. Moreover, suppose that the solution
to the matrix problem with jump $\mathbf{V}^{\left( 2\right) }(k)$
restricted to the $\epsilon $ neighborhood of $\pm 1$ has a solution which
satisfies 
\begin{equation}
\mathbf{M}_{\pm 1}(k)=\mathbf{I}+\frac{1}{b(t)^{\alpha }}\frac{\mathbf{M}%
_{\pm 1}}{k-\left( \pm 1\right) }+O(b(t)^{-\beta }),\ \ \ \left\vert
k-\left( \pm 1\right) \right\vert >\varepsilon  \label{M+-}
\end{equation}%
with some $\alpha \in \lbrack \beta /2,\beta )$. Then the solution $\mathbf{M%
}^{\left( 3\right) }(k)$ to RHP\#3 is given by%
\begin{equation}
\mathbf{M}^{\left( 3\right) }(k)=\mathbf{I}+\frac{1}{b(t)^{\alpha }}\left( 
\frac{\mathbf{M}_{1}}{k-1}+\frac{\mathbf{M}_{-1}}{k+1}\right)
+O(b(t)^{-\beta }),t\rightarrow +\infty ,  \label{M3 asympt}
\end{equation}%
where the error term depends on the distance of $k$ to $\gamma _{\pm
j},j=1,2,3,4$. To obey $\left( \ref{symmetry cond for mat}\right) $ we must
have%
\begin{equation*}
\mathbf{M}_{-1}=-\sigma _{1}\mathbf{M}_{1}\sigma _{1}
\end{equation*}%
and for $\left( \ref{M3 asympt}\right) $ we finally have%
\begin{equation*}
\mathbf{M}^{\left( 3\right) }(k)\sim \mathbf{I}+\frac{1}{b(t)^{\alpha }k}%
\left( \mathbf{M}_{1}-\sigma _{1}\mathbf{M}_{1}\sigma _{1}\right) ,\ \ \
k\rightarrow \infty ,t\rightarrow +\infty .
\end{equation*}%
Consequently the (row) solution to RHP\#2 is then given by%
\begin{equation}
\mathbf{m}^{\left( 2\right) }(k)=%
\begin{pmatrix}
1 & 1%
\end{pmatrix}%
\left\{ \mathbf{I}+\frac{1}{b(t)^{\alpha }}\left( \frac{\mathbf{M}_{1}}{k-1}-%
\frac{\sigma _{1}\mathbf{M}_{1}\sigma _{1}}{k+1}\right) \right\}
+O(b(t)^{-\beta }),t\rightarrow +\infty .  \label{m^2 asym}
\end{equation}%
Thus, the problem boils down to finding $\mathbf{M}_{1}$.

\section{Asymptotics of $\mathbf{B}_{\pm}$ around $k=1$\label{sec asym of
B's}}

In this section we derive necessary asymptotics of jump matrices $\mathbf{B}%
_{\pm }\left( k\right) $ at $k=1$. Observing first that since $k=1$ is a
stationary point for the phase $\Phi \left( k\right) $, we have%
\begin{equation}
\xi \left( k\right) ^{2}\sim \exp \mathrm{i}\left\{ -16t+24t\left(
k-1\right) ^{2}\right\} ,\ \ k\rightarrow 1.  \label{asympt for xi}
\end{equation}%
The common substitution $k=1+\left( 48t\right) ^{-1/2}z$ shifts our cross $%
\Gamma _{1}$ to the origin and removes time dependence from the oscillatory
exponential $\left( \ref{asympt for xi}\right) $. Omitting straightforward
computations based on $\left( \ref{asympt for xi}\right) $, Theorems \ref%
{Thm on Q} and \ref{thm:part} one then gets (below $k=1+\left( 48t\right)
^{-1/2}z,\ \ \ z\rightarrow 0$)%
\begin{equation}
\begin{tabular}{lllll}
\textbf{$B$}$_{-}(k)^{-1}$ & $\sim $ & $%
\begin{pmatrix}
1 & -R_{1}\varphi \left( z\right) ^{2}\mathrm{e}^{-\mathrm{i}z^{2}/2} \\ 
0 & 1%
\end{pmatrix}%
$ & $=:\mathbf{V}_{1}\left( z\right) ,$ & $\arg z=-$\textrm{$i$}$\pi /4$ \\ 
\textbf{$B$}$_{+}(k)$ & $\sim $ & $%
\begin{pmatrix}
1 & 0 \\ 
R_{2}\varphi \left( z\right) ^{-2}\mathrm{e}^{\mathrm{i}z^{2}/2} & 1%
\end{pmatrix}%
$ & $=:\mathbf{V}_{2}\left( z\right) ,$ & $\arg z=$\textrm{$i$}$\pi /4$ \\ 
\textbf{$B$}$_{+}(k)$ & $\sim $ & $%
\begin{pmatrix}
1 & -R_{3}\varphi \left( z\right) ^{2}z^{-2}\mathrm{e}^{-\mathrm{i}z^{2}/2}
\\ 
0 & 1%
\end{pmatrix}%
$ & $=:\mathbf{V}_{3}\left( z\right) ,$ & $\arg z=3$\textrm{$i$}$\pi /4$ \\ 
\textbf{$B$}$_{-}(k)^{-1}$ & $\sim $ & $%
\begin{pmatrix}
1 & 0 \\ 
R_{4}\varphi \left( z\right) ^{-2}z^{-2}\mathrm{e}^{\mathrm{i}z^{2}/2} & 1%
\end{pmatrix}%
$ & $=:\mathbf{V}_{4}\left( z\right) ,$ & $\arg z=-3$\textrm{$i$}$\pi /4$%
\end{tabular}%
\ \ \ ,  \label{Asympt of B}
\end{equation}%
where%
\begin{equation*}
\varphi \left( z\right) =\exp \frac{\mathrm{i}}{2\pi }\left\{ \log \left(
48a^{2}t\right) \log z-\log ^{2}z\right\}
\end{equation*}%
and ($A_{0}$ is given by $\left( \ref{A_0}\right) $)%
\begin{equation*}
\begin{array}{cc}
R_{1}=\mathrm{e}^{16\mathrm{i}t}\dfrac{1+\mathrm{i}\kappa }{1-\mathrm{i}%
\kappa }A_{0}^{2}\exp \dfrac{1}{\pi \mathrm{i}}\left\{ \log a\log \left(
48t\right) +\frac{1}{4}\log ^{2}\left( 48t\right) \right\} , & R_{2}=1/R_{1}
\\ 
R_{3}=48a^{2}t\ R_{1}, & R_{4}=48a^{2}t/R_{1}%
\end{array}%
,
\end{equation*}

\section{Matrix RHP on one cross\label{sec RHP on a cross}}

In this section we formulate (but not solve) a matrix RHP to find the
constant matrix $\mathbf{M}_{1}$ in $\left( \ref{m^2 asym}\right) $. Note
also that since each ray in the small cross $\Gamma _{+1}$ is of a fixed
length $\varepsilon $, the image of $\Gamma _{+1}$ under the change of
variable $z=\left( 48t\right) ^{1/2}\left( k-1\right) $ extends to infinity
as $t\rightarrow +\infty $. This prompts introducing a new infinite cross $%
\Gamma _{0}=\cup _{j=1}^{4}\Gamma _{j}$ where%
\begin{align*}
\Gamma _{1}& =\{u\mathrm{e}^{-\mathrm{i}\pi /4},\,u\in \lbrack 0,\infty )\}
& \Gamma _{2}& =\{u\mathrm{e}^{\mathrm{i}\pi /4},\,u\in \lbrack 0,\infty )\}
\\
\Gamma _{3}& =\{u\mathrm{e}^{3\mathrm{i}\pi /4},\,u\in \lbrack 0,\infty )\}
& \Gamma _{4}& =\{u\mathrm{e}^{-3\mathrm{i}\pi /4},\,u\in \lbrack 0,\infty
)\},
\end{align*}%
the enumeration of $\Gamma _{j}$ being chosen to match that of \cite{GT09}
for easy reference. Orient $\Gamma _{0}$ such that the real part of $z$
increases in the positive direction (see Figure~\ref{fig:contourcross}). 
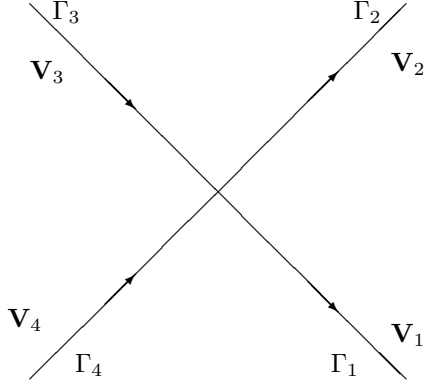
\begin{figure}[tbp]
\begin{picture}(7,5.2)
		
		\put(1,5){\line(1,-1){5}}
		\put(2,4){\vector(1,-1){0.4}}
		\put(4.7,1.3){\vector(1,-1){0.4}}
		\put(1,0){\line(1,1){5}}
		\put(2,1){\vector(1,1){0.4}}
		\put(4.7,3.7){\vector(1,1){0.4}}
		
		\put(5.0,0.1){$\Gamma_1$}
		\put(5.3,4.8){$\Gamma_2$}
		\put(1.3,4.8){$\Gamma_3$}
		\put(1.6,0.1){$\Gamma_4$}
		
		\put(5.8,0.5){$\mathbf{V}_{1}$}
		\put(5.8,4.1){$\mathbf{V}_{2}$}
		\put(1,4){$\mathbf{V}_{3}$}
		\put(0.7,.7){$\mathbf{V}_{4}$}
		
	\end{picture}
\caption{Contours of one cross}
\label{fig:contourcross}
\end{figure}
We are now ready to state our next RHP.

\textbf{RHP\#4 (Matrix RHP on one infinite cross)}. Find a $2\times2$ matrix
valued function \textbf{$M$}$^{\left( 4\right) }\left( z\right) $ which is
analytic away from $\Gamma_{0}$ and satisfies:

(1)\ The jump condition%
\begin{equation}
\mathbf{M}_{+}^{\left( 4\right) }\left( z\right) =\mathbf{M}_{-}^{\left(
4\right) }\left( z\right) \mathbf{V}^{\left( 4\right) }\left( z\right) \text{
for }z\in\Gamma_{0},
\end{equation}
where ($\mathbf{V}_{j}$'s are defined by $\left( \ref{Asympt of B}\right) $)%
\begin{equation}
\mathbf{V}^{\left( 4\right) }(z)=\mathbf{V}_{j}\left( z\right) ,z\in
\Gamma_{j},j=1,2,3,4;
\end{equation}

(2)\ The normalization condition%
\begin{equation*}
\mathbf{M}^{\left( 4\right) }\left( z\right) \sim \mathbf{I},z\rightarrow
\infty ,z\notin \Gamma _{0}.
\end{equation*}

Note that we still use consequently numbering of our jump matrices which
will be reset in the next section.

\section{Rescaling and second conjugation. Model RHP\label{sec model RHP}}

In this section we follow \cite{GT09} to reduce RHP\#4 to the one with the
jump on the real line.

Rewrite the jump matrices $\mathbf{V}_{j}$ as follows%
\begin{align}
\mathbf{V}_{1} & =%
\begin{pmatrix}
1 & -\overline{\rho}\varphi\left( z\right) ^{2}\mathrm{e}^{-\mathrm{i}%
z^{2}/2} \\ 
0 & 1%
\end{pmatrix}
, & \mathbf{V}_{2} & =%
\begin{pmatrix}
1 & 0 \\ 
\rho\mathrm{e}^{\mathrm{i}z^{2}/2}/\varphi\left( z\right) ^{2} & 1%
\end{pmatrix}
,  \notag \\
\mathbf{V}_{3} & =%
\begin{pmatrix}
1 & -\overline{\rho}\varphi\left( z\right) ^{2}\mathrm{e}^{-\mathrm{i}%
z^{2}/2}/\left( \varepsilon z^{2}\right) \\ 
0 & 1%
\end{pmatrix}
, & \mathbf{V}_{4} & =%
\begin{pmatrix}
1 & 0 \\ 
\rho\mathrm{e}^{\mathrm{i}z^{2}/2}/\left( \varepsilon\varphi\left( z\right)
^{2}z^{2}\right) & 1%
\end{pmatrix}
,
\end{align}
where%
\begin{equation*}
\rho=\mathrm{e}^{-16\mathrm{i}t}\dfrac{1-\mathrm{i}\kappa}{1+\mathrm{i}%
\kappa }A_{0}^{-2}\exp\dfrac{\mathrm{i}}{\pi}\left\{ \log a\log\left(
48t\right) +\frac{1}{4}\log^{2}\left( 48t\right) \right\} ,
\end{equation*}
is a unimodular number,%
\begin{equation*}
\varphi\left( z\right) =\exp\frac{\mathrm{i}}{2\pi}\left\{ \log\frac {1}{%
\varepsilon}\log z-\log^{2}z\right\} ,
\end{equation*}
and%
\begin{equation*}
\varepsilon=1/\left( 48a^{2}t\right)
\end{equation*}
is a small parameter as $t\rightarrow+\infty$.

We now conjugate RHP\#4 following \cite{kt} (see also \cite{DeiftZhou1993}
for more details)). To this end introduce the diagonal matrix%
\begin{equation*}
\mathbf{D}=\left( \varphi\left( z\right) \mathrm{e}^{-\mathrm{i}%
z^{2}/4}\right) ^{\sigma_{3}},
\end{equation*}
triangular matrices%
\begin{align}
\mathbf{D}_{1} & =%
\begin{pmatrix}
1 & \overline{\rho} \\ 
0 & 1%
\end{pmatrix}
, & \mathbf{D}_{2} & =%
\begin{pmatrix}
1 & 0 \\ 
\rho & 1%
\end{pmatrix}
,  \notag \\
\mathbf{D}_{3} & =%
\begin{pmatrix}
1 & -\overline{\rho}/\varepsilon z^{2} \\ 
0 & 1%
\end{pmatrix}
, & \mathbf{D}_{4} & =%
\begin{pmatrix}
1 & 0 \\ 
-\rho/\varepsilon z^{2} & 1%
\end{pmatrix}%
\end{align}
and four sectors%
\begin{align}
\Omega_{1} & =\left\{ z|\arg z\in\left( -\pi/4,0\right) \right\} , & 
\Omega_{2} & =\left\{ z|\arg z\in\left( 0,\pi/4\right) \right\} ,  \notag \\
\Omega_{3} & =\left\{ z|\arg z\in\left( 3\pi/4,\pi\right) \right\} , & 
\Omega_{4} & =\left\{ z|\arg z\in\left( -\pi,-3\pi/4\right) \right\} .
\end{align}
Consider the new matrix

\begin{equation}
\mathbf{M}\left( z\right) =\mathbf{M}^{\left( 4\right) }\left( z\right) 
\mathbf{D}\left( z\right) \left\{ 
\begin{array}{cc}
\mathbf{D}_{j}\left( z\right) , & z\in\Omega_{j} \\ 
\mathbf{I}, & \text{else}%
\end{array}
\right. .  \label{second conj}
\end{equation}
Arguing along the same lines as \cite{kt} (see also \cite{DeiftZhou1993} for
more detail) one can show that the substitution $\left( \ref{second conj}%
\right) $ removes the jumps across all $\Gamma_{j}$ but introduces a jump
across the real line instead. Indeed, since $\mathbf{M}^{\left( 4\right)
}\left( z\right) $ has no jump across the real line we have%
\begin{equation*}
\mathbf{M}_{+}\left( z\right) =\mathbf{M}^{\left( 4\right) }\left( z\right)
\left\{ 
\begin{array}{cc}
\mathbf{D}_{+}\left( z\right) \mathbf{D}_{3}\left( z\right) , & z<0 \\ 
\mathbf{D}_{+}\left( z\right) \mathbf{D}_{2}, & z>0%
\end{array}
\right. ,
\end{equation*}%
\begin{equation*}
\mathbf{M}_{-}\left( z\right) =\mathbf{M}^{\left( 4\right) }\left( z\right)
\left\{ 
\begin{array}{cc}
\mathbf{D}_{-}\left( z\right) \mathbf{D}_{4}\left( z\right) , & z<0 \\ 
\mathbf{D}_{-}\left( z\right) \mathbf{D}_{1}, & z>0%
\end{array}
\right. ,
\end{equation*}
and therefore%
\begin{equation*}
\mathbf{M}_{+}\left( z\right) =\mathbf{M}_{-}\left( z\right) \left\{ 
\begin{array}{cc}
\mathbf{D}_{4}\left( z\right) ^{-1}\mathbf{D}_{-}\left( z\right) ^{-1}%
\mathbf{D}_{+}\left( z\right) \mathbf{D}_{3}\left( z\right) , & z<0 \\ 
\mathbf{D}_{1}{}^{-1}\mathbf{D}_{-}\left( z\right) ^{-1}\mathbf{D}_{+}\left(
z\right) \mathbf{D}_{2}, & z>0%
\end{array}
\right. .
\end{equation*}
By a direct computation one shows%
\begin{equation*}
\mathbf{D}_{+}\left( z\right) =\mathbf{D}_{-}\left( z\right) \left\{ 
\begin{array}{cc}
\left( \varepsilon z^{2}\right) ^{\sigma_{3}}, & z<0 \\ 
\mathbf{I}, & z>0%
\end{array}
\right. ,
\end{equation*}
and for the jump matrix we explicitly have%
\begin{equation*}
\mathbf{J}\left( z\right) :=\left( 
\begin{array}{cc}
\varepsilon z^{2} & -\overline{\rho} \\ 
\rho & 0%
\end{array}
\right) ,z<0
\end{equation*}%
\begin{equation*}
\mathbf{J}\left( z\right) :=\left( 
\begin{array}{cc}
0 & -\overline{\rho} \\ 
\rho & 1%
\end{array}
\right) ,z>0.
\end{equation*}
This leads to a new RHP:

\textbf{RHP\#5 (Model RHP)}. Find a $2\times2$ matrix valued function 
\textbf{$M$}$\left( z\right) $ which is analytic on $\mathbb{C\diagdown R}$
and satisfies:

(1)\ The jump condition%
\begin{equation}
\mathbf{M}_{+}\left( z\right) =\mathbf{M}_{-}\left( z\right) \mathbf{J}%
\left( z\right) \text{ for }z\in\mathbb{R},
\end{equation}
where%
\begin{equation}
\mathbf{J}(z)=\left( 
\begin{array}{cc}
\varepsilon z^{2}\mathrm{1}_{\left( -\infty,0\right) } & -\overline{\rho} \\ 
\rho & \mathrm{1}_{\left( 0,\infty\right) }%
\end{array}
\right) ,z\in\mathbb{R};  \label{J}
\end{equation}

(2)\ The normalization condition%
\begin{equation}
\mathbf{M}\left( z\right) \sim \left( \varphi \left( z\right) \mathrm{e}^{-%
\mathrm{i}z^{2}/4}\right) ^{\sigma _{3}},z\rightarrow \infty ,\arg z\in
\left( \pi /4,3\pi /4\right) .  \label{cond 3}
\end{equation}

Just, we are back to the original contour, the real line, but as apposed to
RHP\#0 it is a matrix one now. The relation between RHP\#4 and RHP\#5 is
again given by $\left( \ref{second conj}\right) $.

Recall that in the classical case (i.e. under condition $\left( \ref{sr}%
\right) $) $\left\vert R\left( \pm1\right) \right\vert <1$, which leads to a
constant jump matrix, whereas in our case it is not. A similar situation was
recently considered in \cite{Budylin20} where the first diagonal entry of $%
\left( \ref{J}\right) $ vanishes at $z=0$ linearly. It was shown in \cite%
{Budylin20} that the corresponding RHP can in a certain sense be considered
as a perturbation of the model RHP with a constant jump matrix. This idea
works in our situation too.

\section{Additive RHP\label{sec add RHP}}

In this section we solve RHP\#5 by transforming it to an additive RHP that
is already solved in \cite{Budylin20}. For the reader's convenience we
sketch some steps. Following \cite{kt} (see also \cite{Budylin20},\cite%
{DeiftZhou1993}), introduce a new matrix%
\begin{equation}
\mathbf{N}\left( z\right) :=\left( \frac{\mathrm{d}\mathbf{M}}{\mathrm{d}z}%
\right) \mathbf{M}^{-1}+\frac{\mathrm{i}z}{2}\sigma _{3},  \label{N}
\end{equation}%
where $\mathbf{M}$ is given by $\left( \ref{second conj}\right) $. The
matrix $\mathbf{N}$ is clearly analytic away from $(-\infty ,0]$ and has a
jump across $(-\infty ,0]$. Omitting straightforward computations, one has
an additive RHP:%
\begin{align}
& \mathbf{N}_{+}\left( z\right) -\mathbf{N}_{-}\left( z\right)
\label{Add RHP} \\
& =\frac{2}{z}\mathbf{M}^{\left( 4\right) }\left( 
\begin{array}{cc}
1 & -\overline{\rho }\varphi \left( -z\right) ^{2}\mathrm{e}^{-\mathrm{i}%
z^{2}/2} \\ 
\rho \varphi \left( -z\right) ^{-2}\mathrm{e}^{\mathrm{i}z^{2}/2} & -1%
\end{array}%
\right) \left( \mathbf{M}^{\left( 4\right) }\right) ^{-1}\mathrm{1}_{\left(
-\infty ,0\right) }  \notag \\
& :=\mathbf{E}\left( z\right) ,\ \ \ z\in \mathbb{R},  \notag
\end{align}%
for $\mathbf{N}$ which is considered in \cite{Budylin20} (only some
coefficients are different). We have a less technical approach to $\left( %
\ref{Add RHP}\right) $ (it will be given elsewhere) but here we follow \cite%
{Budylin20}.\ Introducing an error $O\left( \varepsilon ^{1/2}\log
\varepsilon \right) $ we replace $\left( \ref{Add RHP}\right) $ with%
\begin{equation}
\mathcal{N}_{+}\left( z\right) -\mathcal{N}_{-}\left( z\right) =\mathcal{E}%
\left( z\right) ,  \label{Additive RHP}
\end{equation}%
where%
\begin{align}
\mathcal{E}\left( z\right) & =\frac{2}{z}\left( 
\begin{array}{cc}
1 & -\overline{\rho }\varphi \left( -z\right) ^{2}\mathrm{e}^{-\mathrm{i}%
z^{2}/2} \\ 
\rho \varphi \left( -z\right) ^{-2}\mathrm{e}^{\mathrm{i}z^{2}/2} & -1%
\end{array}%
\right) \mathrm{1}_{\left( -\infty ,0\right) }  \label{E} \\
& =\frac{2}{z}\mathrm{1}_{\left( -\infty ,0\right) }\sigma _{3}+\frac{2}{z}%
\left( 
\begin{array}{cc}
0 & -\overline{\rho }\varphi \left( -z\right) ^{2}\mathrm{e}^{-\mathrm{i}%
z^{2}/2} \\ 
\rho \varphi \left( -z\right) ^{-2}\mathrm{e}^{\mathrm{i}z^{2}/2} & 0%
\end{array}%
\right) \mathrm{1}_{\left( -\infty ,0\right) }.  \notag
\end{align}%
We then split $\mathcal{E=E}_{+}-\mathcal{E}_{-}$ where $\mathcal{E}_{\pm }$
and functions that can be analytically continued into $\mathbb{C}^{\pm }$ as
follows (recalling our notation $z_{\pm }=z\pm \mathrm{i}0$)%
\begin{align}
\mathcal{E}_{\pm }\left( z\right) & =-\frac{\mathrm{i}}{\pi }\frac{\log
z_{\pm }}{z}\sigma _{3}  \label{E+-} \\
& -\frac{\mathrm{i}}{z}\frac{1}{\pi }\int_{-\infty }^{0}\frac{z+\mathrm{i}}{%
s-z_{\pm }}\left( 
\begin{array}{cc}
0 & -\overline{\rho }\varphi \left( -s\right) ^{2}\mathrm{e}^{-\mathrm{i}%
s^{2}/2} \\ 
\rho \varphi \left( -s\right) ^{-2}\mathrm{e}^{\mathrm{i}s^{2}/2} & 0%
\end{array}%
\right) \frac{\mathrm{d}s}{s+\mathrm{i}}.  \notag
\end{align}

Substituting $\left( \ref{E+-}\right) $ in $\left( \ref{Additive RHP}\right) 
$ leads to the conclusion that the function $N_{+}\left( z\right) -\mathcal{E%
}_{+}\left( z\right) $ is analytic on $\mathbb{C}$ except for $z=0$ where it
may have a simple pole (this is due to discontinuity of $\mathbf{J}\left(
z\right) $ at $z=0$). That is, the matrix function%
\begin{align*}
& \left( \frac{\mathrm{d}\mathbf{M}}{\mathrm{d}z}\right) \mathbf{M}^{-1}+%
\frac{\mathrm{i}z}{2}\sigma _{3}+\frac{\mathrm{i}}{\pi }\frac{\log z}{z}%
\sigma _{3} \\
& +\frac{\mathrm{i}}{z}\int_{0}^{\infty }\frac{z+\mathrm{i}}{s+z}\frac{%
\mathbf{F}\left( s\right) }{s-\mathrm{i}}\frac{\mathrm{d}s}{\pi }+\frac{%
\mathbf{A}}{z},
\end{align*}%
where%
\begin{equation*}
\mathbf{F}\left( s\right) :=\left( 
\begin{array}{cc}
0 & -\overline{\rho }f\left( s\right) \\ 
\rho \overline{f\left( s\right) } & 0%
\end{array}%
\right) ,f\left( s\right) :=\varphi \left( s\right) ^{2}\mathrm{e}^{-\mathrm{%
i}s^{2}/2}
\end{equation*}%
is entire for some constant matrix $\mathbf{A}$. As in the classical case,
this leads to a first order matrix differential equation. The latter can be
rewritten as a system of second order ODEs that are perturbations of Bessel
equations. As opposed to the the classical case we cannot solve them
explicitly but can asymptotically, which leads to%
\begin{align}
\mathbf{M}\left( z\right) & \sim \left( \mathbf{I}+\dfrac{\mathrm{i}}{z}%
\left( 
\begin{array}{cc}
0 & -\beta \\ 
\overline{\beta } & 0%
\end{array}%
\right) \right) \left( \varphi \left( z\right) \mathrm{e}^{-\mathrm{i}%
z^{2}/4}\right) ^{\sigma _{3}},  \label{M asympt} \\
z& \rightarrow \infty ,\ \ \ \arg z\in \left( \pi /4,3\pi /4\right) ,\ \ \
t\rightarrow +\infty ,  \notag
\end{align}%
where%
\begin{equation}
\beta =\sqrt{\nu }\overline{\rho }\exp \mathrm{i}\left\{ \pi /4+\arg \Gamma
\left( \mathrm{i}\nu \right) \right\} .  \label{betta}
\end{equation}%
Here $\Gamma $ is the Euler gamma function, $\nu =\nu \left( \varepsilon
\right) =-\left( 1/2\pi \right) \log \varepsilon $, and $\varphi \left(
z\right) $ is given by%
\begin{equation*}
\varphi \left( z\right) =\exp \left\{ \mathrm{i}\nu \left( \varepsilon
\right) \log z+\frac{1}{2\pi \mathrm{i}}\log ^{2}z\right\} \text{.}
\end{equation*}

\section{Proof of the Main Theorem \label{sec proof}}

With all our preparation in order, to proof Theorem \ref{Thm on WvN} we
merely need to do a chain of back substitutions to $\left( \ref{q(x,t)}%
\right) $. The end product is available in \cite{GT09}:%
\begin{equation*}
\left. q\left( x,t\right) \right\vert _{x=-12t}\sim \sqrt{\frac{4}{3t}}\func{%
Im}\beta \left( t\right) ,\ \ \ t\rightarrow +\infty ,
\end{equation*}%
where $\beta $ is given by $\left( \ref{betta}\right) $:%
\begin{align*}
\beta \left( t\right) & =\sqrt{\nu \left( t\right) }\overline{\rho }\left(
t\right) \exp \mathrm{i}\left\{ \frac{\pi }{4}+\arg \Gamma \left( \mathrm{i}%
\nu \left( t\right) \right) \right\} \\
& \sim \sqrt{\nu \left( t\right) }\overline{\rho }\left( t\right) \exp 
\mathrm{i}\left( \nu \left( t\right) \log \nu \left( t\right) -\nu \left(
t\right) \right) ,t\rightarrow +\infty .
\end{align*}%
Here $\nu \left( t\right) =\left( 1/2\pi \right) \log \left( 48a^{2}t\right) 
$ and we have used the well-known asymptotic%
\begin{equation*}
\arg \Gamma \left( \mathrm{i}y\right) \sim y\log y-y-\pi /4,\ \ \
y\rightarrow +\infty .
\end{equation*}%
Since

\begin{equation*}
\overline{\rho }=\mathrm{e}^{16\mathrm{i}t}\dfrac{1+\mathrm{i}\kappa }{1-%
\mathrm{i}\kappa }A_{0}^{2}\exp \dfrac{1}{\mathrm{i}\pi }\left\{ \log a\log
\left( 48t\right) +\frac{1}{4}\log ^{2}\left( 48t\right) \right\} ,
\end{equation*}%
where $A_{0}$ is given by $\left( \ref{A_0}\right) $, it follows from $%
\left( \ref{betta}\right) $ that%
\begin{align*}
\left. q\left( x,t\right) \right\vert _{x=-12t}& \sim 2\left( 3t\right)
^{-1/2}\func{Im}\beta \left( t\right) \\
& =\sqrt{\frac{4\nu \left( t\right) }{3t}}\func{Im}\left\{ \overline{\rho
\left( t\right) }\exp \mathrm{i}\left[ \nu \left( t\right) \log \nu \left(
t\right) -\nu \left( t\right) \right] \right\} \\
& =\sqrt{\frac{4\nu \left( t\right) }{3t}}\sin \left\{ 16t-\nu \left(
t\right) \left[ 1+\pi \nu \left( t\right) -\log \nu \left( t\right) \right]
+\delta \right\} ,
\end{align*}%
where%
\begin{align*}
\delta & =\frac{2}{\pi }\int\limits_{0}^{1}\log \left\vert \frac{%
T(s)/T^{\prime }\left( 1\right) }{1-s}\right\vert ^{2}\frac{ds}{1-s^{2}}-%
\frac{2}{\pi }\int\limits_{0}^{1}\frac{\log s{\mathrm{d}s}}{s-2} \\
& +\frac{1}{\pi }\log \frac{\kappa ^{3}+\kappa }{2}\log \frac{\kappa
^{3}+\kappa }{8}-\frac{\pi }{3}+\arctan \frac{2\kappa }{1-\kappa ^{2}}.
\end{align*}%
Thus, the theorem is proven.

\section{Concluding remarks\label{sec concl}}

In this final section we make some remarks on what we have not done.

1. We state our main result for just one ray, $x=-12t$, and not for a wider
region as it is typically done. The main reason is that otherwise we would
not be able to rely directly on \cite{Budylin20}. We fill in this gap
elsewhere by offering an independent of \cite{Budylin20} approach.

2. We believe that Theorem \ref{Thm on WvN} holds for initial data of the
form $\left( \ref{our data}\right) $ but supported on $\left( 0,\infty
\right) $. However, RHP\#0 for this case is no longer regular since $%
m^{\left( +\right) }(k,x)$ has now a simple pole at $\pm 1$ which means that 
$\mathbf{m}(k,x,t)$ has a first order singularity and thus RHP\#0 is
singular. However this issue does not arise if in place of (\ref{R basic
scatt identity}) we use the left scattering identity%
\begin{equation}
T(k)\psi ^{\left( +\right) }(x,k)=\overline{\psi ^{\left( -\right) }(x,k)}%
+L(k)\psi ^{\left( -\right) }(x,k),\text{ }k\in \mathbb{R},
\label{left basic scat ident}
\end{equation}%
where $T$ is the same as above and $L$ is the left reflection coefficient.
But we know that $L=R$ in our case. As before $\left( \ref{left basic scat
ident}\right) $ leads to a new row vector%
\begin{equation}
\mathbf{m}(k,x,t)=\left\{ 
\begin{array}{c@{\quad}l}
\begin{pmatrix}
T(k)m^{\left( +\right) }(k,x,t) & m^{\left( -\right) }(k,x,t)%
\end{pmatrix}%
, & \func{Im}k>0, \\ 
\begin{pmatrix}
m^{\left( -\right) }(-k,x,t) & T(-k)m^{\left( +\right) }(-k,x,t)%
\end{pmatrix}%
, & \func{Im}k<0,%
\end{array}%
\right.
\end{equation}%
where, as before, $m^{\left( \pm \right) }(k;x,t)=$\textrm{$e$}$^{\mp 
\mathrm{i}kx}\psi ^{\left( \pm \right) }(x,t;k)$, which satisfies the jump
condition%
\begin{equation}
\mathbf{m}_{+}\left( k\right) =\mathbf{m}_{-}\left( k\right) \mathbf{V}%
\left( k\right) \text{ for }k\in \mathbb{R},
\end{equation}%
where%
\begin{equation}
\mathbf{V}\left( k\right) =\left( 
\begin{array}{cc}
1-\left\vert R\left( k\right) \right\vert ^{2} & -\overline{R}\left(
k\right) \xi \left( k\right) ^{2} \\ 
R\left( k\right) \xi \left( k\right) ^{-2} & 1%
\end{array}%
\right) .  \label{jump matrix left}
\end{equation}%
Note that the only difference between $\left( \ref{jump matrix left}\right) $
and $\left( \ref{jump matrix}\right) $ is the sign of the power of $\xi
\left( k\right) $ which is the opposite. The well-posedness of the new RHP\
is a more complicated question but it does hold. The next difference is of
course the signature diagram (Figure 1)\ where all signs change to the
opposite which leads to a different conjugation step. Instead of $T_{0}$ we
conjugate our RHP with $T/T_{0}$ and in place of $\left( \ref{B -}\right) $
and $\left( \ref{B+}\right) $ we have

\begin{equation*}
\mathbf{B}_{-}(k)=\left\{ 
\begin{array}{ccc}
\begin{pmatrix}
1 & 0 \\ 
-\dfrac{R(k)T_{0-}(k)^{2}\xi\left( k\right) ^{-2}}{T\left( k\right)
^{2}T\left( k\right) T\left( -k\right) } & 1%
\end{pmatrix}
& , & k\in\mathbb{R}\backslash\left[ -1,1\right] \\ 
\begin{pmatrix}
1 & \dfrac{R(-k)T\left( k\right) ^{2}}{T_{0}\left( k\right) ^{2}\xi\left(
k\right) ^{-2}} \\ 
0 & 1%
\end{pmatrix}
& , & k\in\left[ -1,1\right]%
\end{array}
\right. ,
\end{equation*}%
\begin{equation*}
\mathbf{B}_{+}(k)=\left\{ 
\begin{array}{ccc}
\begin{pmatrix}
1 & -\dfrac{T\left( k\right) ^{2}R(-k)\xi\left( k\right) ^{2}}{T_{0+}\left(
k\right) ^{2}T\left( k\right) T\left( -k\right) } \\ 
0 & 1%
\end{pmatrix}
& , & k\in\mathbb{R}\backslash\left[ -1,1\right] \\ 
\begin{pmatrix}
1 & 0 \\ 
\dfrac{R(k)T_{0}\left( k\right) ^{2}\xi\left( k\right) ^{-2}}{T\left(
k\right) ^{2}} & 1%
\end{pmatrix}
& , & k\in\left[ -1,1\right]%
\end{array}
\right. .
\end{equation*}
Continuing the same way as we did above, we get to a model RHP with a more
complicated jump matrix. We plan to come back to it elsewhere.

3. The general case of WvN resonances (i.e. $q$'s are of the form $\left( %
\ref{pure WvN}\right) $ with arbitrary $\gamma $) hinges on understanding
the zero energy behavior of scattering data. Note that this problem was
first explicitly mentioned in \cite{Klaus91}. The paper \cite{Yafaev81}
suggests that this is a very difficult problem in general. We should
emphasize in this connection that as apposed to the short-range scattering
theory, its WvN type counterpart is not understood to the level suitable to
do the IST. It is to say that there is no "one stop shopping" like \cite%
{Deift79} or \cite{March86} available for facts related to long-range
direct/inverse scattering.

\section{Acknowledgment}

The author acknowledges partial support from NSF under grant DMS-2307774. We
are also thankful to Deniz Bilman for stimulating discussions and the
referees for valuable comments.

\end{document}